\documentclass[a4paper]{amsart}
\usepackage{amsgen,amsmath,amstext,amsbsy,amsopn,amsfonts,amssymb}
\usepackage{amsmath,amssymb,amsfonts,amsthm,enumerate}
\usepackage{amsthm}
\usepackage[latin1]{inputenc} 
\usepackage{color}
\usepackage[pagebackref]{hyperref}
\usepackage[alphabetic]{amsrefs}
\usepackage{xcolor}
\usepackage{tikz-cd}

\usepackage{rotating}

\newtheorem{theorem}{Theorem}
\newtheorem{lemma}[theorem]{Lemma}

\newtheorem{proposition}[theorem]{Proposition}
\newtheorem{obs}[theorem]{Observation} \newtheorem{defi}[theorem]{Definition}

\newenvironment{definition}{\begin{defi}\rm}{\end{defi}}
\newtheorem{exa}[theorem]{Example}

\newtheorem{rem}[theorem]{Remark}

\newtheorem{rems}[theorem]{Remarks}

\newtheorem*{acknowlegment}{Acknowlegment}

\def\B{\mathcal B}

\def\E{\mathcal E}

\def\NN{{\mathbf N}}

\def\RR+{{\mathbf R}^*}

\def\Q_p{{\mathbf Q}_p}

\def\ind{{\rm Ind}}
\def\sgn{\rm sgn}

\def\tous{\qquad\text{for all}\quad}

\def\ind{{\rm Ind}}

\def\Hom{{\mathrm {Hom}}}

\newcommand{\Ind}{\operatorname{Ind}}

\def\tout{\qquad\text{for all}\quad}

\newcommand{\Iso}{\textnormal{Iso}}

\def\Q{\mathbf Q}

\newcommand{\Mod}[1]{\ (\mathrm{mod}\ #1)}
\newtheorem{thmx}{Theorem}
 
\newtheorem{corx}[thmx]{Corollary}
\begin{document}

\title[The $L_p$-dual space of  a semisimple  Lie group]{The $L_p$-dual space of a semisimple  Lie group}

\address{Bachir Bekka \\ Univ Rennes \\ CNRS, IRMAR--UMR 6625\\
Campus Beaulieu\\ F-35042  Rennes Cedex\\
 France}
 \date{\today}
\email{bachir.bekka@univ-rennes1.fr}
\author{Bachir Bekka}

\thanks{The author acknowledges the support  by the ANR (French Agence Nationale de la Recherche)
through the project Labex Lebesgue (ANR-11-LABX-0020-01) .}
\begin{abstract}
Let $G$ be a  semisimple  Lie group. We describe
the irreducible representations of $G$ by linear isometries
on $L_p$-spaces for $p\in (1,+\infty)$ with $p\neq 2.$
More precisely, we show that, for every  such representation $\pi,$
there exists  a parabolic subgroup $Q$ of $G$ such that $\pi$ is equivalent to the 
natural representation of $G$ on $L_p(G/Q)$ 
twisted by a unitary character of  $Q.$
When $G$ is of real rank one, we give a complete classification of 
the possible irreducible representations of $G$ on an $L_p$-space for $p\neq 2,$ up to equivalence.

\end{abstract}
\maketitle
\section{Introduction}
 Let $G$ be a locally compact topological group 
 and $E$ a Banach space. Denote by $\B(E)$ the algebra of bounded operators
  on $E$, equipped   with the strong operator topology, that is, the
 weakest topology  for which the map $\B(E) \to E, T\mapsto Tv$ is continuous  for  every $v\in V.$
An \emph{isometric representation} (for short, a \emph{representation})  of $G$ on $E$ is a continuous group homomorphism 
 $\pi: G\rightarrow\Iso(E)$,  where $\Iso(E)$ is the subgroup of $GL(E)$ consisting of the  linear surjective isometries on $E$. 
    
In the context of  semisimple Lie groups, there is strong evidence (see \cite{HC1}, \cite{Godement1}, \cite{Warner1}, \cite{Fell1}) that the right  notions of irreducibility and equivalence of Banach representations are defined as follows.
Recall that to a representation
 $\pi: G\rightarrow\Iso(E)$ of $G$ there is associated an algebra homomorphism $C_c(G) \to \B (E)$,
denoted by $\pi$ again, defined by
$$
\pi(f) \, = \, \int_G f(g) \pi(g) d\mu_G(g) 
\tout 
f \in C_c(G),
$$
where  $C_c(G)$ is the convolution algebra  of continuous functions on $G$
with compact support and $\mu_G$ is a left Haar measure on $G.$
\begin{definition}
 \label{Def1}
(i)  A representation  $\pi: G\rightarrow\Iso(E)$  is \emph{ completely irreducible}
 if the algebra 
 $\pi(C_c(G)) $ is dense in $\B(E)$ for the strong operator topology.
 
 \noindent
 (ii)  Two representations  $\pi_1: G\rightarrow\Iso(E_1)$ and $\pi_2: G\rightarrow\Iso(E_2)$ 
 on Banach spaces $E_1$ and $E_2$ are \emph{Naimark equivalent} (for short, \emph{equivalent}) if, for $i\in\{1,2\},$  there 
there exists a   dense subspace $V_i$ of $E_i$  which
are $\pi_i(C_c(G)$-invariant and a closed injective linear map  $T: V_1\to V_2$
such that 
$$\pi_2(f)Tv  = T\pi_1(f)v \tout  v\in E_1\hskip.2cm 
f \in C_c(G).$$

 \noindent
 (iii) Let $\E$ be a class of Banach spaces. The $\E$-\emph{dual} of $G$, denoted by 
$ \widehat{G}_\E,$
 is the set of  equivalence classes of completely irreducible representations of $G$ on some Banach space $E\in \E;$

  \end{definition}
  When $E$ is a Hilbert space, a representation   $ G\rightarrow\Iso(E)$ is traditionally called a unitary representation;
  a unitary representation $\pi$ is completely irreducible  if and only if $\pi$ is  irreducible (that is, 
  $\{0\}$ and $E$ are the only $G$-invariant closed subspaces of $E$).
  Moreover, unitary representations $\pi_1$ and $\pi_2$ are Naimark equivalent if and only if 
  $\pi_1$ and $\pi_2$ are unitarily equivalent  (see \cite[4.3]{Warner1}); in particular, if $\E$ denotes  the class of Hilbert spaces,
  then  $ \widehat{G}_\E$ coincides with the usual unitary dual $\widehat{G}$ of $G.$
     
   Let $G$ a noncompact connected semisimple linear Lie group.
     For such a group, we will be concerned with $\widehat{G}_{L_p}$ where $L_p$ is the class of $L_p$-spaces   for $p\in [1, +\infty[.$
By an $L_p$-space,   we mean the usual space $L_p(X, \B, \mu)$ of equivalence classes (modulo null sets) of measurable $p$-integrable functions  $f:X\to \mathbf{C},$ , where $\mu$ is a positive $\sigma$-finite measure defined on a standard Borel space 
 $(X,\B).$
The unitary dual space $\widehat{G}= \widehat{G}_{L_2}$ is of course a classical much studied object (see  for instance \cite{Knapp1}, \cite{Warner1}). We will deal here with   $\widehat{G}_{L_p}$ for $p\neq 2.$

Let $\theta$ be a  Cartan involution on the Lie algebra $\mathfrak{g}$ of $G$
and  $\mathfrak{k}$ (respectively $\mathfrak{p}$) the  eigenspace
for the eigenvalue $1$ (respectively $-1$) of $\theta.$
Let  $\mathfrak{a}$ be a maximal commutative subspace of $\mathfrak{p}$
and let $\Sigma\subset \mathfrak{a}^*$ be the corresponding root system, where 
$\mathfrak{a}^*= \Hom_{\mathbf R}(\mathfrak{a}, \mathbf{R})$ is the real dual space
of $\mathfrak{a}.$
Let $\Sigma^+$ be the set of positive roots for an ordering of  $\Sigma$
and $\mathfrak{n}$  the 
sum of the root spaces  $\mathfrak{g}_\alpha$ for $\alpha \in \Sigma^+.$
Let  $K$, $A,$ and $N$ be the subgroups
of $G$  with Lie algebras $\mathfrak{k},  \mathfrak{a}$ and  $\mathfrak{n},$
respectively. Then $K$ is a maximal compact, $A$ an abelian and $N$ a nilpotent subgroup;
moreover, we have an \emph{Iwasawa decomposition} $G= KAN.$

Let  $M$ be  the centralizer of $A$ in $K$.  The group  $MAN$ is a minimal  parabolic subgroup of $G.$
Recall (see \cite[Section V.5]{Knapp1} or  \cite[Chap. I, 1.2]{Warner1}) that a \emph{parabolic subgroup} of $G$ is a closed subgroup containing a conjugate of $MAN.$ Every such  group is conjugate to a 
\emph{standard parabolic subgroup} $Q=Q_F$,  which is  parametrized by a subset $F$ of the set $\Delta$ of simple positive roots; more precisely, let $\mathfrak{a}_F$ be the intersection of  all $\ker \alpha$
for $\alpha \in F$ and $\mathfrak{n}_F$ the sum of the root spaces $\mathfrak{g}_\alpha$
for $\alpha\in \Sigma^+ \setminus \rm{span}(F)$. We have a \emph{Langlands decomposition}
 $$Q=M_Q A_QN_Q,$$
  where $M_Q$ is reductive and contains $M$, 
  $A_Q= \exp(\mathfrak{a}_F)$ 
  and $N_Q= \exp(\mathfrak{n}_F)$.
 Moreover, $M_Q$ and $A_Q$ commute, and both normalize $N_Q$.
 In particular, the standard minimal parabolic parabolic subgroup is given by 
 $Q_\emptyset= MAN.$

  Fix a parabolic subgroup $Q=Q_F$ of $G$  and   a real number $p\in [1,+\infty[.$ 
 We are going to define the $L_p$ analog of the (unitary) principal 
 series representations associated to $Q$.
 
Let $K_Q:= K\cap M_Q$ and choose a Borel fundamental domain $\Omega_Q\subset K$ for the
 coset space $K/K_Q$.
  Then $\Omega_Q$ is also a fundamental domain for $G/Q,$ 
  since  $G= K Q$ and $K\cap Q= K_Q.$
Every  $g\in G$ has a unique decomposition 
 $$g=\kappa(g) \mu(g) \exp(H(g)) n(g) \qquad\text{for}\quad \kappa(g)\in \Omega_Q, \mu(g)\in M_Q,  H(g) \in \mathfrak{a}_F,n(g)\in N_Q.$$
  Let $\mu_{Q}$ be the unique quasi-invariant Borel probability measure
 on $G/ Q$ which is $K$-invariant.
 We identify $G/ Q$ with $\Omega_Q$ and transfer $\mu_Q$ and the natural $G$-action
 on $G/Q$ to $\Omega_Q.$
Let $\rho_{Q,p} \in \mathfrak{a}^*$ be  defined by 
$$\rho_{Q,p} := \dfrac{1}{p}\sum_{\alpha \in \Sigma^+ \setminus \rm{span}(F)}
 \dim(\mathfrak{g}_\alpha) \alpha.$$
 Observe that $\rho_{MAN,2}$ is the familiar half-sum of positive roots.

 Fix a real linear form $\lambda\in \mathfrak{a}^*$
 and a unitary character $$\chi: M_Q\to  \mathbf{S}^1=\mathbf{U}(1)$$ of the reductive group $M_Q.$ 
We define a  representation $\pi (Q,  \chi, \lambda, p)$ of $G$ on 
$$L_p(G/Q, \mu_{Q}) = L_p(\Omega_Q, \mu_{Q})$$
 by 
 $$
 \pi(Q,\chi,\lambda, p)(g) f(x) = \overline{\chi(\mu(g^{-1} x))} e^{-(i\lambda+ \rho_{Q,p}) H(g^{-1} x)} f(\kappa(g^{-1} x))
 $$
 for $f\in L_p(\Omega_Q, \mu_{F}), x\in \Omega_Q, g\in G.$ 
It is easily verified  (see Section~\ref{S:Proof-Theo1}) that $ \pi(Q,\chi, \lambda, p)$
 is indeed an isometric representation of $G$ on  $L_p(G/Q, \mu_{Q})$. 

We denote by $\widehat{M_Q^{\rm Ab}}$ the group of unitary characters  of $M_Q.$
Here is our main result.

\begin{thmx}
\label{Theo1}
 Let $G$  be a noncompact connected semisimple linear Lie group,
 and let  $p\in (1, +\infty)$ with  $p\neq 2.$
 Let $\pi$ be a completely irreducible representation of $G$ on an $L_p$-space. Then there exists
  a parabolic subgroup $Q$ of  $G$,
a real linear form $\lambda\in \mathfrak{a}^*,$ and a unitary character $\chi \in \widehat{M_Q^{\rm Ab}}$
 such that $\pi$ is equivalent to the representation $ \pi(Q,\chi, \lambda, p)$
  on $L_p(G/Q, \mu_{Q})$ described above.
  \end{thmx}

In order to study the irreducibility of a   representation $ \pi(Q,\chi, \lambda, p)$
 as in the previous theorem, we consider its  associated Harish-Chandra $(\mathfrak{g}, K)$-module.
 For a unitary representation $\sigma$ of $M_Q$ and for 
  $\nu\in \mathfrak{a}^*_\mathbf{C}=  \Hom_{\mathbf R}(\mathfrak{a}, \mathbf{C}),$
  following \cite[5.2]{Wallach-book}, 
 we denote  by $I_{Q,  \sigma, \nu}$
  the associated  infinitesimal nonunitary principal  series of $G,$
 that is, the $(\mathfrak{g}, K)$-module underlying  the induced 
representation 
$$ \ind_{Q}^G(\sigma \otimes \nu\otimes 1_{N_Q})=  \Ind_{M_QA_QN_Q}^G (\sigma \otimes \nu\otimes 1_{N_Q}).$$

   \begin{proposition}
 \label{Pro1}
  \begin{itemize}
  \item[(i)] Let $ \pi(Q,\chi, \lambda, p)$ be a representation of $G$ 
  on $L_p(G/Q, \mu_{Q})$ as in Theorem~\ref{Theo1}.
  Then $ \pi(Q,\chi, \lambda, p)$ is completely irreducible if and only if the 
  principal  series representation $I_{Q,  \chi, \nu}$ is (algebraically) irreducible,
  where 
  $$\nu= i\lambda+ \delta_p\rho_{Q,2} $$
for $\delta_{p}=\dfrac{2}{p}- 1\in (-1,1).$
 \item[(ii)] For $i=1,2,$ let  $\pi(Q_i,\chi_i,\lambda_i, p_i)$ 
  be  representations of $G$ 
  on $L_{p_i}(G/Q_i, \mu_{Q_i})$ as in Theorem~\ref{Theo1}.
  Then  $\pi(Q_1,\chi_1,\lambda_1, p_1)$ 
 and $\pi(Q_2,\chi_2,\lambda_2, p_2)$ are equivalent 
 if and only if the corresponding principal  series representations $I_{Q_1,  \chi_1, \nu_1}$
 and $I_{Q_2,  \chi_2, \nu_2}$ as in (i) are (algebraically) equivalent.
 \end{itemize}
 \end{proposition}
 Principal series representations  $I_{Q,  \sigma, \nu}$ are of course central objects in the  representation theory of
 semisimple   Lie groups;  their structure   has been much studied, especially in the case where $Q$ is minimal parabolic, and
 it has been shown that ``most" of them  are  irreducible (see \cite{Gelfand-Naimark}, \cite{Bruhat}, \cite{PRV}, \cite{Kostant1}, 
 \cite{Wallach1}, \cite{Lepowski-Wallach}). 
 
 Using known irreducibility and equivalence results from the literature, we can 
 settle  the case where $G$ is  a  simple Lie group with real rank one. So, $G$ belongs to 
 one of  the series $\mathrm{SO}_0(n,1), \mathrm{SU}(n,1),$  $ \mathrm{Sp}(n,1)$
 or $\mathrm{Sp}(n,1)$ for $n\geq 2$ or $G$ is the exceptional Lie group $F_{4(-20)}$.
  We set $P:= MAN.$ 
 \begin{thmx}
\label{Theo2} 
Let $p\in (1, +\infty)$ with  $p\neq 2.$ 
\begin{itemize}
\item[(i)] Let $G=\mathrm{SO}_0(n,1)$ for $n\geq 2$.  Then $\widehat{M^{\rm Ab}}={1_M}$ for $n\neq 3$
and  $\widehat{M^{\rm Ab}}\cong \mathbf{Z}$ for $n=3.$
The  representation $ \pi(P,\chi, \lambda, p)$ is irreducible for  every $\lambda\in  \mathfrak{a}^*$ and every  $\chi \in \widehat{M^{\rm Ab}}.$
 \item[(ii)] Let $G= \mathrm{SU}(n,1)$ for $n\geq 2$. Then $\widehat{M^{\rm Ab}}\cong \mathbf{Z}$.
\begin{itemize}
\item[(ii1)] The representation  $ \pi(P,\chi, \lambda, p)$  is irreducible 
 for  every $\lambda\in \mathfrak{a}^*\setminus \{0\}$ and every  $\chi \in \widehat{M^{\rm Ab}}.$
\item [(ii2)] After an appropriate  identification of $ \widehat{M^{\rm Ab}}$ with $\mathbf{Z},$    
the representation  $\pi(P,m, 0, p)= \pi(P,\chi, 0, p)$ for $m\in \mathbf{Z}$ is \emph{not} irreducible 
if and only if   $p$ belongs to the finite set 
$$
 \left\{\dfrac{2n}{k}: k\in \{1,\dots,  2n-1\}, k\equiv m\Mod 2, k\neq m, k\neq 2n+m\right\}.
$$
\end{itemize}
\item[(iii)] Let  $G= \mathrm{Sp}(n,1)$  for $n\geq 2$. Then $\widehat{M^{\rm Ab}}={1_M}.$
\begin{itemize}
\item[(iii1)] The  representation $ \pi(P,1_M, \lambda, p)$  is irreducible 
 for  every $\lambda\in \mathfrak{a}^*\setminus \{0\}$ 
 \item[(iii2)] The  representation $\pi(P,1_M,0, p)$  is \emph{not} irreducible if 
 and only if   $p=\dfrac{2n+1}{2n}$ or $p=2n+1.$
 \end{itemize}
\item[(iv)] Let  $G= \mathrm{F}_{4(-20)}.$  Then $\widehat{M^{\rm Ab}}={1_M}.$
\begin{itemize}
\item[(iv1)] The  representation $\pi(P,1_M, \lambda, p)$  is irreducible 
 for  every $\lambda\in \mathfrak{a}^*\setminus \{0\}$ 
 \item[(iv2)] The  representation $\pi(P,1_M, 0, p)$  is \emph{not} irreducible if 
 and only if $p$ belongs to the set 
 $\left\{\dfrac{11}{10 }, \dfrac{11}{9 },  \dfrac{11}{8}, \dfrac{11}{3 }, \dfrac{11}{2 }, 11\right\}.$
 \end{itemize}
\end{itemize}
\end{thmx}
\begin{corx}
\label{Cor2} 
Let $G$ be one of the groups as in Theorem~\ref{Theo2}.
\begin{itemize}
\item[(i)] Let $p\in (1, +\infty)$ with  $p\neq 2.$ 
Every completely irreducible representation of $G$ on an $L_p$-space
is equivalent either to the trivial representation $1_G$ 
or to one of the irreducible representations of the form $\pi(P, \chi, \lambda,p)$
from  Theorem~\ref{Theo2}.
 \item[(ii)] Two irreducible representations $\pi(P, \chi, \lambda,p)$
and   $\pi(P, \chi', \lambda',p')$ from  Theorem~\ref{Theo2} with $(\chi', \lambda', p')\neq (\chi, \lambda,p)$
are equivalent if and only if 
 $$  (\chi', \lambda', p')= (\overline{\chi}, \lambda, p)
  \quad  \text{or} \quad  (\chi', \lambda', p')= (\overline{\chi}, -\lambda, q),$$
where $q$ is the conjugate exponent of $p.$.
\end{itemize}
 \end{corx}
 The paper is organized as follows. 
 The proofs of Theorem~\ref{Theo1}, Proposition~\ref{S:Proof-Prop1},
 and Theorem~\ref{Theo2} are
  is given in Sections~\ref{S:Proof-Theo1},  \ref{S:Proof-Prop1},
  and  \ref{S:Proof-Theo2} respectively.
  As an example, the case of the group $G= SL_2(\mathbf{R}),$ 
   which is a twofold cover of $SO_0(2, 1),$ is treated in
  Section~\ref{S:Exa1}.

    \section{Proof of Theorem~\ref{Theo1}}
 \label{S:Proof-Theo1}
 Let $G$  be a noncompact connected semisimple linear Lie group,
  Let $(X, \B, \mu)$ be a standard Borel space equipped with  a $\sigma$-finite measure $\mu$
  on $\B.$
  Let 
  $$\pi: G\to \Iso(L_p(X, \mu))$$
   be a  completely irreducible representation of $G$   by linear isometries on $L_p(X, \mu)$ for $p\in (1, +\infty)$ with  $p\neq 2.$
   Assume that $\pi$ is not the trivial representation $1_G$ of $G.$
   
   Observe that $L_p(X, \mu)$ is isometrically isomorphic to $L_p(X, \mu')$ for every 
  $\sigma$-finite measure   $\mu'$ on $X$ which is equivalent to $\mu;$ 
    indeed, if $\varphi=  \frac{d\mu'}{d\mu}$
  denotes the Radon-Nikodym derivative of  $\mu'$ with respect to $\mu,$
  then $$f\mapsto f \varphi^{1/p}$$ is a bijective linear isometry from $L_p(X, \mu')$
  to $L_p(X, \mu).$
   So, upon choosing a probability measure
    in the measure class of $\mu,$ we can and will  assume in the sequel that $\mu$
    is a probability measure.

    \vskip.2cm
$\bullet$ \emph{First step.} We claim that there exists a measure-class preserving measurable action 
$$
G\times X\to X, \qquad (g, x)\mapsto gx
$$
of $G$ on $X$ 
and a measurable map   $c: G\times X\to  \mathbf{S}^1$ such that,   for $ f\in L_p(X, \mu), g\in G,$ and  $x\in X$ we have
\begin{equation}
\label{Rep1}
  \pi(g)f(x)= c(g^{-1}, x) \left(\frac{dg^{-1}_*(\mu)}{d\mu}(x)\right)^{1/p}f(g^{-1}x),
\end{equation}
   where $ \mathbf{S}^1$ is the set of complex numbers of modulus 1 and 
   $ g_*(\mu)$ is the image of $\mu$ under the map $x\mapsto gx;$
  moreover $c$ satisfies the cocycle relation
\begin{equation}
\label{Rep2}
c(g_1g_2, x) = c(g_1, g_2x)c(g_2,x) \tous  g_1,g_2\in G,\ \text{for almost all} \ x\in X.
\end{equation}
Indeed,   let $g\in G.$   By  the Banach-Lamperti theorem (see \cite{Banach}, \cite{Lamperti}),
  there exists a  measure-class preserving measurable map 
  $\varphi(g): X\to X$ and a measurable map $c(g): X\to \mathbf{S}^1$ such that 
  $$
  \pi(g)f(x)= c(g)(x) \left(\frac{d\varphi(g)_*(\mu)}{d\mu}(x)\right)^{1/p}f(\varphi(g)(x)) 
  $$
   for $ f\in L_p(X, \mu), g\in G,$ and  $x\in X.$
 Since $\pi(g_1g_2)= \pi(g_1)\pi(g_2)$ for all $g_1, g_2\in G$ and  since $g\mapsto \pi(g) f$ is measurable for every 
 $f\in L_p(X, \mu),$ it is readily checked that $(g,x)\mapsto \varphi(g^{-1})x$ is 
a measure-class preserving $G$-action on $X$ and that 
$$c:(g,x)\to c(g^{-1})x$$ is a measurable
map satisfying  Equation~(\ref{Rep2}).

\vskip.3cm
Recall that another measurable cocycle $c': G\times X\to \mathbf{S^1}$ is said to be cohomologous to $c$ if there exists a measurable map  $b: X\to   \mathbf{S^1}$  such that
$$
c'(g,x)= b(gx)c(g,x) b(x)^{-1} \tous  g\in G,\ \text{for almost all} \ x\in X.
$$
Let  $\pi'$ be  the representation of $G$ on   $L_p(X,\mu)$ defined by the  same formula~(\ref{Rep1}),
with $c'$ instead of $c.$ Then $\pi'$ is equivalent to $\pi;$ indeed, the map
$U: L_p(X,\mu)\to L_p(X,\mu),$ defined by
$$
Uf(x)= b(x)f(x)  \tous  f\in L_p(X,\mu), x\in X,
$$
is a bijective isometry which intertwines $\pi'$ and $\pi.$

 \vskip.2cm
$\bullet$ \emph{Second step.}
 We claim that the measure $\mu$ is continuous (that is, $\mu$ has no atoms). 
 
   Indeed, let $A$ be the (at most countable) set of atoms of  $\mu.$
The decomposition of $\mu$  into its atomic and continuous parts
 is   $\mu= \mu_a+ \mu_c$, where $\mu_a= \mu|_A$ and $\mu_c= \mu-\mu_a.$
 We can write $$L_p(X, \mu)=  L_p(X, \mu_c)\oplus \ell_p(A)$$
 and it follows from the Banach-Lamperti theorem that 
 $ L_p(X, \mu_c)$ are $\ell_p(A)$ are $\pi(G)$-invariant  and so define subrepresentations
 of $\pi.$ Since $\pi$ is  irreducible, it follows that either $\mu$ is purely atomic 
 or $\mu$ is continuous. 
 
 Assume by contradiction that $\mu$ is purely atomic, that is, 
 $L_p(X, \mu)=\ell_p(A)$.
Let 
 $G\curvearrowright A$ be the  action of $G$  on $A$  and 
 $c: G\times A\to \mathbf{S}^1$ the cocycle as above. 
Since $\pi$ is continuous, the map  
 $$G\to \ell_p(A), \qquad g\mapsto c(g, x) \delta_{g\cdot x}$$
  is continuous for every $x\in A.$
As $G$ is connected, this implies that $G$  fixes pointwise  $A$.
 By irreducibility of $\pi,$ it follows that $A$ is a singleton $\{x_0\}.$
   Equation (\ref{Rep2}) shows that   then that $g\mapsto c(g,x_0)$ is 
  homomorphism from $G$ to $\mathbf{S}^1.$ Since $G$ is semisimple,
  $G$ has a trivial abelianization and  hence $c(g,x_0)=1$ for all $g\in G.$ So, 
$\pi=1_G;$  this is a contradiction and the claim is proved.

  \vskip.3cm
  As is well-known (see \cite[Theorem 2.1.19]{Zimmer1}),
   the measurable $G$-system $(X, \mu, G)$ admits a   continuous compact model, that is, there exists a compact metric space $Y$ on which $G$ acts continuously, a probability measure $\nu$ on the Borel subsets of $Y$ and a $G$-equivariant Borel isomorphism  $\Phi: Y\to X$ with $\Phi_*(\nu)= \mu.$
  
 Let  $c': G\times Y\to  \mathbf{S}^1$ be defined by $c'(g, y)= c(g, \Phi(y))$.
 The  representation $\pi',$ defined on $L_p(Y, \nu)$  by  the formula
 $$\pi'(g)f(y)= c'(g^{-1}, y) \left(\frac{dg^{-1}_*(\nu)}{d\nu}(y)\right)^{1/p}f(g^{-1}y),
 $$
 is equivalent to $\pi.$ So, without loss of generality,
 we  may assume  that  $X$ is a compact metric space on which $G$ acts continuously
 and that $\mu$ is a quasi-invariant probability measure  on the Borel subsets of $X.$
 
  \vskip.3cm
    Next, let $K$ be a maximal compact subgroup of $G$  and denote by $dk$ the normalized Haar measure on $K.$
   Then  $\mu'= \int_K k_*(\mu) dk,$ defined by 
   $$
   \mu' (A)=  \int_K \mu(k A)  dk 
   $$
   for all Borel subsets $A$ of $X,$
    is a probability measure on $X$ which is $K$-invariant; moreover,
    $\mu'$ is equivalent to $\mu.$ So, we can and will assume
    in the sequel that $\mu$ is $K$-invariant.
    
 \vskip.2cm
$\bullet$ \emph{Third  step.}
We claim that there are countably many $K$-orbits 
$\mathcal{O}_1, \mathcal{O}_2, \cdots,$ in $X$ such that 
$$\mu\left(\bigcup_{i\geq 1} \mathcal{O}_i\right)=1.$$

Indeed, let $Y:=X/K$ be the space of  $K$-orbits, equipped with the quotient topology
structure given by the quotient map $r: X\to Y$ and with the probability measure $\nu:=r_*(\mu).$

Since $K$ is compact, $Y$ is a Hausdorff space. So, $Y$ is a second countable  compact  
space and is therefore  a standard   Borel space.
Let $\nu= \nu_a+ \nu_c$ be the decomposition of $\nu$  into its atomic and continuous parts.
  The claim will be proved if we show that $\nu_c=0.$
  
By a well-known disintegration lemma  (see e.g. Lemma 11.1  \cite{MackeyInd} or Theorem 5.14 in \cite{EinsiedlerWard}),
 there exists a map  $$\theta:Y\to {\rm Prob}(X), y\mapsto \theta_y,$$
  where $ {\rm Prob}(X)$ is the set of Borel probability measures on $X,$
 with the following properties:
 \begin{itemize}
 \item[(a)] for every $y\in Y,$ we have $\theta_y( r^{-1}(y))=1;$
 \item[(b)] for every $f\in L_1(X,\mu),$ the map $y\mapsto \int_X f(x)d\theta_y(x) $ is measurable
 and we have
 $$\int_X f d\mu= \int_Y  \left(\int_X f(x) d\theta_y(x)\right) d\nu(y).$$
 \end{itemize}
 Moreover, if  a second map $\theta':Y\to {\rm Prob}(X)$ satisfies (a) and (b), then 
 $\theta_y=\theta'_y$ for $\nu$-almost every $y\in Y.$
 
 Since $\mu$ is $K$-invariant, it follows from the uniqueness of $\theta$ that
  $\theta_y$ is  the unique $K$-invariant probability measure on the $K$-orbit $y\subset X,$
  for $\nu$-almost every $y.$
  
 For $f\in C(X)$ and $y\in Y,$  let $f|_y\in C(y)$ denote the restriction 
 of $f$ to $y$. The map  $f\mapsto (y\mapsto f_y)$ extends
 to an isometric  isomorphism  
 $$S: L_p(X, \mu)\to \int^{\oplus}_Y L_p(y,\theta_y)  d\nu(y),$$ 
 where  $\int^{\oplus}_Y L_p(y,\theta_y)  d\nu(y)$ is the $L_p$-direct integral 
 of  the family  $(L_p(y,\theta_y))_{y\in Y}$ in the sense of \cite{DeJeu-Roz}. 
 Let $\pi_y$ denote the representation of $K$ on $L_p(y, \theta_y)$ given by 
  $$\pi_y(k)f(x)= c(k^{-1},x) f(k^{-1}x)) \tout f\in L_p(y, \theta_y),  k\in K,  x\in y.$$
  The  direct integral  $\int^{\oplus}_Y \pi_y d\nu(y)$
   of the $\pi_y$'s is a representation of $K$ defined on  $\int^{\oplus}_Y L_p(y,\theta_y)  d\nu(y)$
   and the  map  $S$ intertwines 
   $\pi|_K$ with $\int^{\oplus}_Y \pi_y d\nu(y)$; for this, see  \cite[Theorem 4.9]{DeJeu-Roz}
   (the proof given there extends to representations twisted by a cocycle as in our situation).  
  
    For  $\delta \in \widehat{K}$ (recall that $\widehat{K}$ is the unitary dual of $K$), set 
$$\chi_\delta=  d(\delta) {\rm ch}_\delta,$$
where ${\rm ch}_\delta= \mathrm{Tr}\circ \delta$
is the usual character of $\delta$,
and define
$$ P_\pi(\delta) :=\pi(\overline{\chi_\delta})= \int_{K} \overline{\chi_\delta}(k) \pi(k) dk.
$$
Then $P_\pi(\delta)$ is a continuous projection of  $L_p(X, \mu)$  onto 
 the isotypical $K$-submodule of  $L_p(X, \mu)$  of type $\delta.$
Similarly, for every $y\in Y,$ 
$$P_{\pi_y}(\delta):= \pi_y(\overline{\chi_\delta})= \int_{K} \overline{\chi_\delta}(k) \pi_y(k) dk$$
is a continuous projection of  $L_p(y, \theta_y)$  onto 
the isotypical  $K$-submodule of  $L_p(y, \theta_y)$ of  type $\delta.$

The field $y\to P_{\pi_y}(\delta)$ defines a decomposable operator $\int^{\oplus}_Y P_{\pi_y}(\delta) d\nu(y)$ on 
 $\int^{\oplus}_Y L_p(y,\theta_y)  d\nu(y)$ and, since $S$ intertwines 
   $\pi|_K$ and $\int^{\oplus}_Y \pi_y d\nu(y),$ we have 
\begin{equation}
  \label{Rep3}
  SP_\pi(\delta)S^{-1}=  \int^{\oplus}_Y P_{\pi_y}(\delta) d\nu(y)
  \end{equation}
  
   Let  $\Omega$ be a Borel subset  of $Y$. Denote by  $V_\Omega$
   the range of   the multiplication operator on $\int^{\oplus}_Y L_p(y,\theta_y)  d\nu(y)$ by $\mathbf{1}_{\Omega}$, that is, 
  $$
  \begin{aligned}
  V_\Omega&= \left\{f=(f_y)_{y\in Y}\in \int^{\oplus}_Y L_p(y,\theta_y)  d\nu(y):  \mathbf{1}_{\Omega}f= f\right\}\\
  &=
\left\{f=(f_y)_{y\in Y}\in \int^{\oplus}_Y L_p(y,\theta_y)  d\nu(y):  f_y= 0\quad \text{for} \quad \nu \text{-almost every}\quad y\notin \Omega\right\}.
\end{aligned}
  $$
  Observe that 
   \begin{equation}
  \label{Rep4}
  \int^{\oplus}_Y P_{\pi_y}(\delta) d\nu(y)
 \quad \text{leaves invariant} \quad V_{\Omega}.
\end{equation}

 Assume, by contradiction, that the continuous part  $\nu_c$ of $\nu$ is non-zero;
  so, denoting by $Y_c$ the complement in $Y$ of set of  atoms of $\nu,$ we have
  $\nu(Y_c)>0$ and  the measure  $\nu|_{Y_c}$ is continuous. 
  
 For  $\delta \in \widehat{K}$, the subset 
 $$
 D_\delta:= \{y\in Y: P_{\pi_y}(\delta)=0\}
 $$
 of $Y$ is measurable. Indeed, let  $(f_n)_n$ be a dense sequence in $C(X)$ for the uniform
 convergence; then, by Urysohn lemma,  $(f_n|_y)_{n}$ is dense in $C(y)$ and hence in
 $L_p(y,\theta_y)$ for every $y\in Y,$ and so
  $$
D_\delta=  \bigcap_{n\in \mathbf{N}}\{y\in Y:   P_{\pi_y}(\delta) (f_n|_y)=0\}.$$
Since $y\mapsto  P_{\pi_y}(\delta) (f_n|_y)$ belongs to  $\int^{\oplus}_Y L_p(y,\theta_y)  d\nu(y)$ 
 and is therefore a measurable field, $D_\delta$ is measurable.
  
  Observe that  $\widehat{K}$ is countable, since $K$ is a separable compact group.
  So, the set  
  $$D:=\bigcap_{\delta \in \widehat{K}} D_\delta$$
   is measurable.
  We claim that 
\begin{equation}
\label{Rep5}
\nu(D)= 0.
\end{equation}
 Indeed, assume, by contradiction that $\nu(D)>0.$ 
 Then $V_D\neq \{0\}.$ However, for every $(f_y)_{y\in Y}\in V_{D},$ 
 we have 
  $$
P_{\pi_y}(\delta)(f_y)=0  \tout \delta \in \widehat{K}
$$
and hence $f_y=0$ for every $y\in D.$ This contradiction shows that (\ref{Rep5}) holds.

It follows from (\ref{Rep5}) that there exists  $\delta \in \widehat{K}$ such
that    $\nu\left(Y_c \cap N_\delta\right)< \nu(Y_c).$ So, for 
$$
\Omega:= Y_c \setminus  N_\delta
$$
 we have  $\nu(\Omega)>0$.
 Since  $\nu|_{Y_c}$ is continuous,  we can therefore  find a  partition
 $$\Omega= \coprod_{n\in \mathbf{N}} \Omega_n$$   of $\Omega$ in countably many Borel subsets $\Omega_n$ with $\nu(\Omega_n)>0$ for all $n\in \mathbf{N}.$
 We have 
  $$ \left(\int^{\oplus}_Y P_{\pi_y}(\delta) d\nu(y)\right)V_{\Omega_n}\neq \{0\} \tout n\in \mathbf{N}.
 $$ 
 since $P_{\pi_y}(\delta)\neq 0$  for every $y\in \Omega.$  Hence, for every $n\in \mathbf{N},$ we can find
 $$f_n\in  \left(\int^{\oplus}_Y P_{\pi_y}(\delta) d\nu(y)\right)V_{\Omega_n} \qquad\text{with} \qquad f_n \neq 0.$$ 
 Since $\Omega_n\cap \Omega_m=\emptyset $ for $n\neq m,$
 it follows from (\ref{Rep4}) that the family $(f_n)_{n\in \mathbf{N}}$ is linearly independent.
 Therefore,  on the one hand, the family $(S^{-1}f_n)_{n\in \mathbf{N}}$ of elements in $L_p(X,\mu)$ 
 is linearly independent.
 On the other hand, it follows from (\ref{Rep3}) that 
 $$S^{-1}f_n\in P_\pi(\delta) (S^{-1}(V_{\Omega_n}))\subset P_\pi(\delta) (S^{-1}(V_{\Omega}))\tout n\in\mathbf{N}.$$
 This implies that the subspace $ P_\pi(\delta)(S^{-1}(V_{\Omega}))$ of $L_p(X,\mu)$ is infinite dimensional.
So,  the projection $P_\pi(\delta)$ does not have  finite-dimensional range. This is not possible 
since  $\pi$ is completely irreducible (see \cite[Theorem 2]{Godement1} or \cite[Lemma 33]{HC1}).
 This is a contradiction shows that $\nu$ is an atomic measure.

\vskip.2cm
$\bullet$ \emph{Fourth  step.}
We claim that there exists a point $x_0\in X$ such that $\mu(Gx_0)=1$. Indeed,
it follows from the third step that there exists  a point $x_0\in X$ such that $\mu(Kx_0)>0$
and therefore $\mu(Gx_0)>0.$
If we show that the $G$-action of $(X, \mu)$ is ergodic, then it will follow that 
$\mu(G x_0)=1.$

Assume, by contradiction, that there exists a partition $X= X_1\coprod X_2$ into two  
$G$-invariant Borel subsets $X_i$
with $\mu(X_i)>0$ for $i=1,2.$ Then 
$$L_p(X, \mu)= V_{X_1}\oplus V_{X_2},$$ where
$$V_{X_i}= \{f\in L_p(X,\mu): \mathbf{1}_{X_i} f= f\} \cong L_p(X_i, \mu|_{X_i});$$
moreover,  Formula~(\ref{Rep1}) shows that  $V_{X_i}$
is $\pi(G)$-invariant.
Since $V_{X_i}\neq \{0\}$  for $i=1,2$, this contradicts the fact that $\pi$ is irreducible
and the claim is proved.

 \vskip.3cm
The stabilizer of $x_0\in X$ as above is a closed subgroup $H$.
Hence, upon disregarding a set of measure zero,
 we may and will identify  $X$, as \emph{Borel $G$-space,} with  the space $G/H$, equipped  with its quotient Borel structure.
 
 Next, we draw some  consequences for the  cocycle $c: G\times X\to \mathbf{S^1}$
 over the transitive $G$-space $X$.
 First, upon passing to a cohomologous  cocycle, we can assume that 
 $c$ is a strict cocycle, that is, Equation (\ref{Rep2}) holds for all $g_1, g_2\in G$ and $x\in X.$
 Moreover, choose a Borel section  $s: G/H\to G$ for the projection $G\to G/H$
 with $s(H)=e.$ Observe that $g s(x) s(gx)^{-1}$ belongs to $H,$  for every $g\in G$ and $x\in X.$ 
 Then there exists a continuous homomorphism $\chi: H\to \mathbf{S^1}$
  such that $c$ is cohomologous to the cocycle $c_\chi:G\times X\to \mathbf{S^1}$
given by 
  $$
c_\chi(g,x)= \chi( g s(x) s(gx)^{-1}) \tout g\in G, x\in X;
  $$
for all this, see Theorem 5.27 in  \cite{Varadarajan1}.

We may and will  in the sequel assume that the cocycle $c$ coincides with $c_\chi.$
Also, as is well-known (see e.g. \cite[Theorem 1.1]{MackeyInd}),  all 
 $G$-quasi-invariant $\sigma$-finite Borel measures on $G/H$  are mutually equivalent;
so, we may assume that  the $G$-quasi-invariant  $\mu$ on $G/H$ 
is   one of the standard $G$-invariant  Borel measures on $G/H$ constructed
in \cite[\S.1]{MackeyInd}.

 \vskip.2cm
$\bullet$ \emph{Fifth step.}
 We claim that $K$ acts transitively on the space $G/H.$
 
 Indeed, on the one hand, since $G$ is a connected Lie group, $X=G/H$ is a $C^\infty$- smooth 
 connected manifold; moreover,  the measure $\mu$ on $X$ is locally Lebesgue measure, that is, 
$\mu$ is given by a smooth density times Lebesgue measure in any
local coordinates on $X;$ indeed,  this follows from the
fact that the maps $x\mapsto gx$ are diffeomorphisms 
of $X$ for every $g\in G.$

On the other hand, it follows from the third step that there exists a   $K$-orbit $\mathcal{O}$ in $X$
with $\mu(\mathcal{O})>0.$ As is well-known, $\mathcal{O}$ is a submanifold of $X$, diffeomorphic to $K/L$
where $L$ is the stabilizer in $K$ of a point in $K$ (see e.g. \cite[III, \S1, 7.]{Bourbaki1}).
In view of what we said about $\mu,$ this  is possible only if $\dim \mathcal{O}= \dim X,$ that is, 
if $\mathcal{O}$ is open in $X.$ Since $X$ is connected, it follows that $\mathcal{O}=X$.

\vskip.2cm
$\bullet$ \emph{Sixth step.}
We claim that $H$ is contained as a finite index subgroup in a parabolic subgroup of $G$, as defined in the Introduction.
Indeed, it follows from the fifth step and   Lemma~\ref{Cocompact} below, that there exists a 
parabolic subgroup $Q$ which contains $H$ and which has   the following properties; let  
$Q=M_Q A_QN_Q$ be the Langlands decomposition of $Q$. Write  $Q^0= CSA_QN_Q$  for the connected identity component $Q^0$ of $Q,$ where $C$ is  the maximal compact  factor of $M_Q^0$ and $S$ is the product of all the non-compact simple factors  of $M_Q^0$. Then 
$$H^0= C'SA_QN_Q,$$
 where  $C'$ is a connected closed subgroup of $C.$

We claim $H$ has finite index in $Q.$ Assume, by contradiction,  that $H$ has infinite index in $Q.$ Observe that $Q$ is a direct product $Q= Z Q^0$ for a finite subgroup
$Z$ of its center (see \cite[Theorem 7.53]{Knapp2}). So, $H$ is of the form
$H= T SA_QN_Q$  for a subgroup $T$ of  infinite index in the compact  group $L:=ZC.$

Let $\chi$ be the unitary character of $H$ associated to $c.$  
Observe that  $\chi$ is trivial on the semisimple Lie group $S$. 
Moreover, we have  $[\mathfrak{a}_F, \mathfrak{n}_F]=  \mathfrak{n}_F$
for the Lie algebras $\mathfrak{a}_F$ and $\mathfrak{n}_F$  of  $A_Q$ and $N_Q$,
where $F$ is the set of simple roots associated to $Q$ as in the Introduction;
it follows that  the commutator subgroup of $A_QN_Q$ coincides with 
$N_Q$  and hence  that $\chi$ is also trivial on $N_Q$.

Choose  a Borel fundamental domain $\Omega_0\subset G$ for $G/Q$
and a fundamental domain $\Omega_1\subset L$ for $L/T.$ 
Then $\Omega_0\Omega_1$ is a fundamental domain for $G/H$
and every $g\in G$ has  unique decompositions 
 $$
 g=\omega_0(g) q(g)=\omega_0(g)\omega_1(g) h(g) \qquad\text{for} \quad\omega_i(g)\in \Omega_i, q(g)\in Q,
 h(g)\in H.
 $$
 Let  $\nu_1$ be the $L$-invariant probability measure on $\Omega_1\cong L/T.$
Identify $L_p(G/H, \mu)$ with $L_p(\Omega_0\times\Omega_1,\nu_0\otimes \nu_1)$
for a $G$-quasi-invariant probability measure $\nu_0$ on $\Omega_0\cong G/Q.$
For $g\in G$ and $F\in L_p(\Omega_0\times\Omega_1,\nu_0\otimes \nu_1),$ we have
$$
\pi(g) F (x)= \chi(h(g^{-1} x))  F(\omega_0(g^{-1} x)\omega_1(g^{-1} x)) \tout x\in \Omega_0\times\Omega_1.
$$
Define a representation $\pi_Q$ 
of $Q$ on $L_p(\Omega_1, \nu_1)$ by 
 $$
 \pi_Q(q) f(x_1) = \chi(h(q^{-1} x_1))  f(\omega_1(q^{-1} x_1))
 $$
 for $f\in L_p(\Omega_1, \nu_1), x_1\in \Omega_1, q\in Q.$ 
  
 Since $\Omega_1\cong Q/H$ is infinite,   $L_p(\Omega_1, \nu_1)$ is infinite dimensional;
therefore,  the restriction $\pi_Q|_L$ of $\pi_Q$ to  the compact group $L$ is not
 irreducible (see \cite[4.2.2.4]{Warner1}).  Hence, there exists a closed $\pi_Q(L)$-invariant 
 subspace $V$ of $L_p(\Omega_1, \nu_1)$ with $V\neq \{0\}$ and $V\neq L_p(\Omega_1, \nu_1).$
 Observe that $\pi_Q(q)$   is a multiple of $\chi(q)$ for $q\in A_Q$
 and $\pi_Q(q)$ is the identity for $q\in SN_Q.$ Since $Q= LSA_QN_Q,$ it follows that 
 $V$ is $Q$-invariant.

For $F\in L_p(\Omega_0\times\Omega_1,\nu_0\otimes \nu_1)$ and $x_0\in \Omega_0,$ 
let $F_{x_0}\in  L_p(\Omega_1,\nu_1)$ be defined by 
$$F_{x_0}(x_1)= F(x_0x_1).$$
Let $\widetilde{V}$ be subspace of $ L_p(\Omega_0\times\Omega_1,\nu_0\otimes \nu_1)$ of all 
$F\in  L_p(\Omega_0\times\Omega_1,\nu_0\otimes \nu_1)$ such that 
$F_{x_0}\in V$ for $\nu_0$-almost every $x_0\in \Omega_0.$
Then $\widetilde{V}$ is a proper non trivial closed subspace of  $L_p(\Omega_0\times\Omega_1,\nu_0\otimes \nu_1).$

We claim that $\widetilde{V}$ is  $\pi(G)$-invariant.
Indeed, let $F\in \widetilde{V}, g\in G$ and  $F'= \pi(g) F;$  
 for $x_0\in \Omega_0, x_1\in \Omega_1,$ set  $x=x_0x_1$
 and let $q=q(x_0^{-1}g)$ be the $Q$-component of $x_0^{-1}g.$ We have
 $\omega_0(g^{-1} x)=\omega_0(g^{-1} x_0)$ and 
 $$
\begin{aligned}
F'_{x_0} (x_1)&=F'(x_0 x_1)=  \chi(h(g^{-1} x))  F(\omega_0(g^{-1} x_0)\omega_1(g^{-1} x))\\
&= \chi(h(((x_0^{-1}g)^{-1}x_1))  F_{\omega_0(g^{-1} x_0)}(\omega_1((x_0^{-1}g)^{-1} x_1 ))\\
&= \chi(h(q^{-1} x_1))  F_{\omega_0(g^{-1} x_0)}(\omega_1(q^{-1} x_1)))\\
&=  (\pi_Q(q)  F_{\omega_0(g^{-1} x_0)})(x_1) 
\end{aligned}
$$
Since $F_{\omega_0(g^{-1} x_0)}\in V$ for almost every $x_0$ and since $V$ is $Q$-invariant,
it follows that $F'\in \widetilde{V}$ and the claim is proved.

Now, $\pi$ is irreducible. So, we have obtained a contradiction which shows that 
$H$ has finite index in $Q.$

\vskip.2cm
$\bullet$ \emph{Seventh step.}
We claim  that $H=Q.$ Indeed, $H$ contains $Q^0,$ since
 $H$ has finite  index in $Q.$  As mentioned in the sixth step,
$Q$ is a direct product $Q=ZQ^0$ for a finite abelian subgroup $Z$.
Hence, we have  $Q= Z'H$ for a subgroup $Z'$ of $Z.$ 
We have to show that $Z'$ is trivial. 
 
Assume, by contradiction,  that $Z'\neq \{e\}.$
Then there exists  a non trivial character $\delta$ of $Z'.$
We proceed as in the sixth  step. 
Let  $\Omega_0\subset G$ a fundamental domain for $G/Q$. We identify $L_p(G/H, \mu)$ with $L_p(\Omega_0\times Z',\nu_0\otimes \nu_1)$
for a convenient $G$-quasi-invariant probability measure $\nu_0$ on $\Omega_0\cong G/Q$
and   a  probability measure $\nu_1$ on $Z'.$

For $F\in L_p(\Omega_0\times Z',\nu_0\otimes \nu_1)$ and $x_0\in \Omega_0,$ 
let $F_{x_0}\in  L_p(\Omega_1,\nu_1)$ be defined by $F_{x_0}(x_1)= F(x_0x_1).$
Let $\widetilde{V}_\delta$ be subspace of $ L_p(\Omega_0\times\Omega_1,\nu_0\otimes \nu_1)$ of all 
$F$ such that 
$F_{x_0}(z x_1)= \delta(z) F_{x_0}( x_1)$ for all $z\in Z'$ and $\nu_0$-almost every $x_0\in \Omega_0.$
Then $\widetilde{V}_\delta$ is a proper non trivial closed subspace of  $L_p(\Omega_0\times\Omega_1,\nu_0\otimes \nu_1).$
Moreover, one checks that $\widetilde{V}_\delta$ is $\pi(G)$-invariant; this contradicts the irreducibility of $\pi.$

\vskip.2cm
$\bullet$ \emph{Eighth step.} There exists a
a real linear form $\lambda\in \mathfrak{a}^*$ and a unitary character $\chi$  of $M_Q$ 
 such that $\pi$ is equivalent to the representation $ \pi(Q,\chi, \lambda, p)$
  on $L_p(G/Q, \mu_{Q})$ described in Theorem~\ref{Theo1}.
  
  Indeed, by the seventh step, we have $H= Q=M_QA_QN_Q.$  The unitary character of $Q$
  associated to the cocycle $c$ is trivial on $N_Q$ and $M_Q$ centralizes $A_Q;$
  hence,  this character is of the form $\chi \otimes \chi'$ for  unitary characters $\chi$  of $M_Q$ 
  and $\chi'$ of $A_Q.$ Let $F\subset \Sigma$ be  the set  of simple positive roots  associated to $Q$ 
  There exists $\lambda\in \mathfrak{a}^*$ such that 
  $$\chi'(\exp X)= e^{-i \lambda(X)} \tout X\in \mathfrak{a}_F= \log (A_Q).$$
   Let $\mu$ be the unique quasi-invariant Borel probability measure
 on $G/ Q$ which is $K$-invariant.  For  $g\in G,$ decompose $g$ according to 
 $G= KM_QA_QN_Q$ as 
  $$g= \kappa(g) \mu(g) \exp(H(g)) n(g) \quad\text{for}\quad \kappa(g)\in \Omega_Q, \mu(g)\in M_Q,  H(g) \in  \mathfrak{a}_F,n(g)\in N_Q.$$
  The Radon-Nikodym derivative of $g_* \mu$ with respect to $\mu$  is given by 
  (see e.g. \cite[Proposition 8.44]{Knapp2})
  $$
  \frac{dg_*(\mu)}{d\mu}(x)= e^{-\rho_F(H(g))}
  $$
  for 
  $$
 \rho_F= \sum_{\alpha \in \Sigma^+ \setminus \rm{span}(F)}
 \dim(\mathfrak{g}_\alpha) \alpha.
 $$
 This finishes the proof of Theorem~\ref{Theo1}.

   \vskip.2cm
   The following lemma, which was used in the proof above, is a consequence of the
   description of cocompact subgroups of  semisimple Lie groups from \cite{Witte}
   (see also   \cite{Goto-Wang}).
 \begin{lemma}
 \label{Cocompact}
 Let $G$ be a noncompact connected semisimple linear Lie group
 with maximal compact subgroup $K.$ Let $H$ be a closed subgroup of
 $G$ with the property that the natural action of $K$ on $G/H$ is transitive.
 Upon replacing $H$ by one of its conjugate,  there exists a standard parabolic subgroup $Q$ which contains $H$ and 
 which has   the following properties;
denote by   $Q=M_Q A_QN_Q$ the  Langlands decomposition of $Q;$ 
 write  $Q^0= CSA_QN_Q$ for the  identity component $Q^0$ of $Q,$
  where $C$ is  the maximal compact  factor of $M_Q^0$ and $S$ is the product of all the non-compact simple factors  of $M_Q^0$. Then there exists a  connected closed subgroup $C'$ of $C$ such that 
  $$H^0= C'SA_QN_Q.$$
  
 \end{lemma}
 \begin{proof}
 Since $K$ acts transitively on $G/H$, we have $G= KH.$
 In particular, $H$ is a cocompact subgroup of $G.$
By   \cite[Theorem 1.2]{Witte} and upon conjugating $H$,   there exists  a  standard parabolic subgroup $Q$ of $G$
 with   the following properties:
 \begin{itemize} 
 \item $Q=M_Q A_QN_Q$ contains $H;$
 \item there exist a connected, normal subgroup $X$ of $S$ and 
 a connected closed subgroup $Y$ of $CA_Q$
 such that $H^0= YXN_Q$, where $S$ and $C$ are 
 as in the statement of the Lemma.
\end{itemize} 
 Since $H/H^0$ is at most countable, we have
 $$G= \bigcup_{n\geq 1} KH_0 g_n$$
  for some sequence $(g_n)_{n\geq 1}$ of elements in $G.$
 As $K$ is compact and $H_0$ is closed, $KH_0 g_n$ is a closed subset of $G$ for every $n.$
 Hence, by Baire's category theorem, $KH_0$ has a non empty interior.
 By homogeneity, it follows that  $KH_0$ is open in $G$. 
 Since $G$ is connected,  we have therefore 
 $G= KH_0,$ that is,
 \begin{equation}
\label{Rep6}
G= KYXN_Q.
\end{equation}

 Recall that the multiplication map $ K\times A\times N \to G$  
 from the Iwasawa decomposition $G= KAN$ is a diffeomorphism. 
 The reductive Lie group $M_Q$ has also an Iwasawa decomposition
 $$M_Q= K_1A_1 N_1,$$  
 for $K_1:= K\cap M_Q;$ we have  $ A=A_1 A_Q$ as a direct product and  $N= N_1 N_Q$
 as semi-direct product (see \cite[Proposition7.82]{Knapp2}).
  In particular, we have $C\subset K_1.$
 Let $\Omega\subset K$ be a fundamental domain for $K/K_1.$
 So,  
 \begin{equation}
\label{Rep7}
 \text{the product map} \qquad  \Omega\times  K_1\times A_1\times A_Q \times  N_1 \times N_Q \to G\qquad\text{is a bijection.}
\end{equation}
 The Iwasawa decomposition of the semisimple Lie group $X\subset S$ is 
$X=  K_2 A_2N_2,$ where  $K_2:= K_1\cap X$ and where  $A_2$ and $N_2$ are subgroups 
of  $A_1$ and $N_2$. 
By    (\ref{Rep6}), we have therefore
$$G= KYXN_Q = K  K_2 Y A_2N_2N_Q.
$$
Since $Y\subset CA_Q,$  it follows  from  (\ref{Rep7})
that 
$$\dim A_2= \dim A_1,\quad \dim N_2= \dim N_1, \quad\text{and} \quad A_Q\subset Y.$$
So, we have $S=X$ and $Y=C' A_Q$ for  a connected closed subgroup $C'$ of $C.$ 

 \end{proof}
 
  \section{Proof of Proposition~\ref{Pro1}}
 \label{S:Proof-Prop1}
 Let $Q=M_QA_Q N_Q$ be a  parabolic subgroup   of $G$;
 for a real linear form $\lambda\in \mathfrak{a}^*$
 and a unitary character $\chi: M_Q\to  \mathbf{S}^1$,
 consider the    representation $\pi (Q, \lambda, \chi, p)$ of $G$ on 
 $L_p(G/Q, \mu)$  as in the Introduction.

 We may realize  $\pi (Q, \lambda, \chi, p)$ in the so-called compact picture as follows.
 Let $C(K, \chi)$ be the vector space  of  continuous functions
 $F: K\to \mathbf{C}$ such that 
$$F(xk)= \overline{\chi(k)} F(x) \tout k\in K\cap M_Q,\  x\in K;$$
equip $C(K, \chi)$ with the norm
$$
\Vert F\Vert_p= \left(\int_{K/ (K\cap M_Q)} |F(x)|^p d\mu(x))\right)^{1/p},
$$
where $\mu$ is the unique $K$-invariant probability measure on $K/ (K\cap M_Q).$
Let $L_p(K, \chi)$ be the completion of $(C(K, \chi), \Vert\cdot\Vert_p).$ 
Let  $\rho_{Q,p}\in  \mathfrak{a}^*$ be defined
as in the Introduction.  Then   $\pi (Q, \lambda, \chi, p)$ acts on  $L_p(K, \chi)$ through
 $$
 \pi(Q,\chi,\lambda, p)(g) F(x) = \overline{\chi(\mu(g^{-1} x))} e^{-(i\lambda+ \rho_{Q,p}) H(g^{-1} x)} F(\kappa(g^{-1} x))
 $$
 for $F\in L_p(K, \chi), x\in K$ and $g\in G$ with  decomposition 
 $$g=\kappa(g) \mu(g) \exp(H(g)) n(g), \qquad \kappa(g)\in K, \mu(g)\in M_Q,  H(g) \in \log (A_Q),n(g)\in N_Q$$
 Observe that, for $g\in K,$ the operator $\pi(Q,\chi,\lambda, p)(g)$ is simply left translation on  $L_p(K, \mu)$ by $g^{-1}$.
  
 Set $E:=L_p(K, \chi)$. Recall that the subspace   $E_K$  of $K$-finite vectors in $E$ is the space of functions $F\in E$
 such that 
 $$\pi(Q,\chi,\lambda, p)(K)F=\left\{\pi(Q,\chi,\lambda, p)(g)F: g\in K\right\} $$ 
 spans a finite-dimensional subspace of $E.$  The space  $E^\infty$  of  $C^\infty$-vectors in $E$
 consists of smooth functions $F\in C(K, \chi).$
 As a $K$-module, $E^\infty$ is the representation of
 $K$ differentiably induced by $\chi|_{K\cap M_Q}$ (in the sense of \cite[5.3.1]{Warner1}).
 Therefore, by Frobenius reciprocity (which remains valid in this context),
 the multiplicity of $\delta\in \widehat{K}$ in $E^\infty$ is equal to the multiplicity 
 of $\chi$ in $\delta|_{K\cap M_Q}$ and is therefore finite, since $\delta$
 is finite dimensional. Hence,  the multiplicity of $\delta$ in  
 $E_K\cap E^\infty$ is finite. Since $E_K\cap E^\infty$ is dense in $E$
 (see e.g. \cite[Theorem 4.4.5.16]{Warner1}),  it follows that 
 the multiplicity of $\delta$ in $E$ is finite. In other words, 
 $\pi(Q,\chi,\lambda, p)$ is an \emph{admissible} representation of $G$ 
 (as defined \cite[p.207]{Knapp1} and \cite[3.3.5]{Wallach1}) or is a \emph{$K$-finite} representation 
 (as defined in \cite[4.5.1]{Warner1}).

 As is well-known,  $E_K$ is a module  over  the universal enveloping algebra $\mathcal{U}(\mathfrak{g}_\mathbf{C})$ of 
the complexification of the Lie algebra $\mathfrak{g}$ of $G.$  Set 
  \begin{equation}
  \label{Rep9}
  \nu:= i\lambda+ \rho_{Q,p}- \rho_{Q,2}= i\lambda+ \left(\dfrac{2}{p}-1\right)\rho_{Q,2};
  \end{equation}
 then  $E_K$ coincides 
  with  the $(\mathfrak{g}, K)$-module associated to
the  (nonunitary) induced  representation $ \ind_{Q}^G(\chi \otimes \nu\otimes 1_{N_Q})$.
This $(\mathfrak{g}, K)$-module consists of the  functions $f:G\to \mathbf{C}$ which are $C^\infty$  and have the following properties: 
\begin{itemize}
\item $f(gman)=\overline{\chi(m)} e^{-( i\lambda+\nu +\rho_{Q,2})(\log a)} f(g)$  for all  $g\in G$ and  $man\in Q;$
\item the linear span of $K\cdot f$ is finite dimensional space, where 
$K\cdot f$ is the space of translates $ g\mapsto f(k^{-1} g)$ for $k\in K.$
\end{itemize}
 (for this, see \cite[5.2.1]{Wallach1}). 
So, the $(\mathfrak{g}, K)$ module associated to $\pi(Q,\chi,\lambda, p)$
is the principal series representation denoted $I_{Q,  \chi, \nu}$ in \cite{Wallach1}
for $\nu$ as in (\ref{Rep9}).

 Now, since  $ \pi(Q,\chi, \lambda, p)$ is admissible,
  $ \pi(Q,\chi, \lambda, p)$ is  completely irreducible if and only if its 
  associated $(\mathfrak{g}, K)$-module $I_{Q,  \chi, \nu}$ is algebraically irreducible.
Moreover, two such representations    $\pi(Q_1,\chi_1,\lambda_1, p_1)$ and $\pi(Q_1,\chi_1,\lambda_1, p_1)$
  are Naimark equivalent 
 if and only if the $(\mathfrak{g}, K)$-modules $I_{Q_1,  \chi_1, \nu_1}$
 and $I_{Q_2,  \chi_2, \nu_2}$  are algebraically equivalent (for all this, see Theorems 4.5.5.4 and  4.5.5.2 in \cite{Warner1}).

  \section{Proof of Theorem~\ref{Theo2}}
 \label{S:Proof-Theo2}
 We assume from now on that  $G$ is  a  simple Lie group with real rank one.   In this case, we have $\dim \mathfrak{a}=1.$ 
 
 Choose $\alpha \in \Sigma ^+$ and $H\in \mathfrak{a}$  such that $\frac{1}{2}\alpha \notin  \Sigma $
 and  $\alpha(H)=1.$
 We identify $\mathfrak{a}^*_{\mathbf{C}}$ with $\mathbf{C}$
 by means of the map $\lambda\mapsto \lambda(H).$

There is, up to conjugation,   only one parabolic subgroup, namely $P:= MAN.$
For  $p\in (1+\infty),$ we will write $\rho_p$ instead of $\rho_{P,p}$.
In view of  Proposition~\ref{Pro1}, we have to study, given $\lambda\in \mathbf{R}$ and $\chi\in\widehat{M^{\rm Ab}},$
 the irreducibility of the (non unitary) principal series representation $I_{P,  \chi, \nu_{\lambda,p}}$ for  
 $$ \nu_{\lambda,p}= i\lambda+ \left(\dfrac{2}{p}-1\right)\rho_{2}$$
 and to decide when two such  representations are equivalent.
 We will apply several times  the following  result of Kostant  about the case where $\nu=1_M.$ 
 Set 
 $$t_\alpha:=\begin{cases}   (\dim \mathfrak{g}_\alpha)/2\quad  \text{if} \quad 2 \alpha\notin \Sigma\\
  (\dim\mathfrak{g}_\alpha)/2 +1\quad  \text{if} \quad 2\alpha\in\Sigma,
    \end{cases}
  $$
  where $\mathfrak{g}_\alpha\subset \mathfrak{g}$ is the usual  root space corresponding to $\alpha;$
  set  also
  $$n_\alpha:=\begin{cases} 1 \quad  \text{if} \quad 2 \alpha\notin \Sigma\\
 2\quad  \text{if} \quad 2\alpha\in\Sigma.
    \end{cases}
  $$
  Recall that 
  $$
  \rho_2= (\dim \mathfrak{g}_\alpha)/2 + \dim \mathfrak{g}_{2\alpha}.
  $$
For $\nu\in  \mathbf{C},$ we have (see Theorem 2 in \cite{Kostant1} and Theorem 2.9.8 in \cite{Kostant2}):
  \begin{equation}
  \label{Rep10}
  I_{P,  1_M, \nu}\ \text{is \emph{not}\  irreducible} \Longleftrightarrow \nu\notin (-t_\alpha, t_\alpha)\ \text{and} \ \nu+ \rho_2 \in n_\alpha \mathbf{Z}.
  \end{equation}
 We will treat separately the different simple groups $G$ which may occur;
 for the  data we will use concerning these groups, we refer to \cite[Chap. X]{Helgason}.

 \vskip.2cm
$\bullet$ Let $G=\mathrm{SO}_0(n,1)$ for $n\geq 2.$ Here, $K\cong SO(n)$ and $M\cong SO(n-1)$.
Moreover, we have $\dim\mathfrak{g}_\alpha= n-1$ and $2\alpha\notin \Sigma.$ So, 
$$t_\alpha= \dfrac{n-1}{2}, \qquad n_\alpha= 1, \qquad\text{and} \qquad \rho_2= \frac{n-1}{2};$$
hence, for  $\lambda\in\mathbf{R},$ we have
$$\nu_{\lambda, p}= i\lambda+ \left(\dfrac{2}{p}-1\right)\frac{n-1}{2}.$$
Since $\left|\dfrac{2}{p}-1\right|<1,$ it follows from  (\ref{Rep10}) that 
$$I_{P,  1_M,  \nu_{\lambda,p}} \quad \text{is irreducible for all} \ \lambda\in \mathbf{R} \ \text{and} \ p\in (1, +\infty) \ \text{with} \ p\neq 2.$$

The case where $n\neq 3$ is settled since the abelianization of $M$ is  then trivial and so $\widehat{M^{\rm Ab}}=1_M.$

Assume now that $n= 3.$  Observe that in this case $G\cong \mathrm{SL}_2(\mathbf{C})$
and $M \cong \mathrm{U}(1)$. We use  here the results from \cite[\S. 13, B]{Thieleker}:
identifying  $\widehat{M^{\rm Ab}}$ with $\frac{1}{2} \mathbf{Z}$, we have, for $m\in\widehat{M^{\rm Ab}}$
and $\nu\in \mathbf{C},$
 $$
  I_{P,  m, \nu}\ \text{is   irreducible} \Longleftrightarrow(\nu+l+1)(\nu-l-1)\neq 0$$
  for all $l\in \{|m|+k: k=0,1, 2, \cdots\}.$ 
 Since $(n-1)/2= 1$,  this condition is clearly satisfied for $\nu_{\lambda, p}$ and so we obtain that
  $I_{P,  \chi,  \nu_{\lambda,p}}$
  is irreducible for all $\lambda\in \mathbf{R}, \chi\in \widehat{M^{\rm Ab}},$
  and  $p\in (1, +\infty)$ with $ p\neq 2.$
  
  \vskip.5cm
$\bullet$  Let $G=\mathrm{SU}(n,1)$ for $n\geq 2.$
Here, $K\cong \mathrm{SU}(n)$ and $M\cong \mathrm{U}(n-1)$.
Moreover, we have $\dim\mathfrak{g}_\alpha= 2(n-1)$ and  $\dim\mathfrak{g}_{2\alpha}=1$. So, 
$$t_\alpha= n, \qquad n_\alpha= 2, \qquad\text{and} \qquad \rho_2= n;$$
hence, for  $\lambda\in\mathbf{R},$ we have
$$\nu_{\lambda, p}= i\lambda+ \left(\dfrac{2}{p}-1\right)n.$$
Since $\left|\dfrac{2}{p}-1\right|<1,$ it follows from  (\ref{Rep10}) that 
$$I_{P,  1_M,  \nu_{\lambda,p}} \quad \text{is irreducible for all} \ \lambda\in \mathbf{R} \ \text{and} \ p\in (1, +\infty) \ \text{with} \ p\neq 2.$$

 In  \cite[Proposition 1]{Kraljevic} precise necessary and sufficient conditions are given on $\sigma\in \widehat{M}$ and 
 $\nu \in \mathbf{C}$ for $I_{P,  \sigma,  \nu}$ to be irreducible. 
 In the special case of $\sigma \in \widehat{M^{\rm Ab}}$, this criterion reads as follows.
 Identify  $\widehat{M^{\rm Ab}}$ with $\frac{1}{n+1} \mathbf{Z}$ as in \cite{Kraljevic}
 and for $\sigma\in\widehat{M^{\rm Ab}}$
and $\nu\in \mathbf{C},$ define $(s_1, \dots, s_{n+1})\in \mathbf{C}^{n+1}$ by 
$$
s_i=\begin{cases}
 \dfrac{1}{2} (\nu- (n-1)\sigma) \quad  \text{if} \quad  i=1\\
\sigma+\dfrac{n}{2}-i+1 \quad  \text{if} \quad 2\leq i\leq n\\
 -\dfrac{1}{2} (\nu+ (n-1)\sigma) \quad  \text{if} \quad i=n+1.
 \end{cases}
 $$
Then $I_{P,  m, \nu}$ is not irreducible 
if and only if either $s_1-s_i\in \mathbf{Z}\setminus\{0\}$ for all  $2\leq i\leq n$ or 
 $s_{n+1}-s_i\in \mathbf{Z}\setminus\{0\}$ for all  $2\leq i\leq n$.

 It follows immediately  that $I_{P,  \sigma,  \nu_{\lambda,p}}$ is irreducible
 if $\lambda\neq 0.$  So, we may assume that $\lambda=0.$
 
 Writing $\sigma= \dfrac{m}{n+1}$ for $m\in \mathbf{Z},$ the criterion above
 implies   that $I_{P,  \sigma,  \nu_{0,p}}$ is not irreducible if and only if 
 either 
    $$
 \dfrac{1}{2} \left(\nu_{p,0} - (n-1)\dfrac{m}{n+1}\right)-\dfrac{m}{n+1}+\dfrac{n}{2} \in \mathbf{Z}\setminus\{0\}
 $$
  or  
  $$
- \dfrac{1}{2} \left(\nu_{p,0} + (n-1)\dfrac{m}{n+1}\right)-\dfrac{m}{n+1}+\dfrac{n}{2} \in \mathbf{Z}\setminus\{0\}
 $$
 that is,  if and only if 
 $$
 \text{either} \quad \nu_{p,0} -(m+n)  \in 2\mathbf{Z}\setminus\{0\} \quad\text{or}\quad \nu_{p,0} -(m+n)  \in 2\mathbf{Z}\setminus\{0\}
 $$
  As $\nu_{\lambda, 0}=\left(\dfrac{2}{p}-1\right)n,$ we see that $I_{P,  \sigma,  \nu_{0,p}}$ is not irreducible 
 if only if  
 \begin{equation}
  \label{Rep11}
  \text{either} \quad 2\left(\dfrac{1}{p}-1\right)n -m \in 2\mathbf{Z}\setminus\{0\} \quad\text{or}\quad 
   \dfrac{2n}{p}- m \in 2\mathbf{Z}\setminus\{0\}.
   \end{equation}
   It is clear that (\ref{Rep11}) is equivalent to 
   $$
   p\in  \left\{\dfrac{2n}{k}: k\in \{1,\dots,  2n-1\}, k\equiv m\Mod 2, k\neq m, k\neq 2n+m\right\}.
   $$
 \vskip.5cm
$\bullet$ Let $G= \mathrm{Sp}(n,1)$ for $n\geq 2.$
Here, $K\cong  \mathrm{Sp}(n)\times  \mathrm{Sp}(1) $ and $M\cong \mathrm{Sp}(n-1)\times  \mathrm{Sp}(1) $.
Moreover, we have $\dim\mathfrak{g}_\alpha= 4(n-1)$ and  $\dim\mathfrak{g}_{2\alpha}=3$. So, 
$$t_\alpha= 2n-1, \qquad n_\alpha= 2, \qquad\text{and} \qquad \rho_2= 2n+1;$$
hence, for  $\lambda\in\mathbf{R},$ we have
$$\nu_{\lambda, p}= i\lambda+ \left(\dfrac{2}{p}-1\right)(2n+1).$$
 Observe that  the abelianization of $M$ is   trivial and so $\widehat{M^{\rm Ab}}=\{1_M\}.$
Moreover, we have
$$\left|\dfrac{2}{p}-1\right|(2n+1)\geq 2n-1 \Longleftrightarrow p\in (1, \frac{2n+1}{2n}] \cup [ 2n+1, +\infty).$$
and
$$\left(\dfrac{2}{p}-1\right)(2n+1)+ 2n-1  \in 2 \mathbf{Z}
\Longleftrightarrow p\in \left \{\dfrac{2n+1}{k}: k\in \mathbf{\NN^*}\right\}.
$$
 It follows from  (\ref{Rep10}) that 
\begin{itemize}
\item  $I_{P,  1_M,  \nu_{\lambda, p}}$  is irreducible 
 for  every $\lambda\in \mathfrak{a}^*\setminus \{0\}$ 
 \item $I_{P,  1_M,  \nu_{0, p}}$  is not irreducible if 
 and only if  $p=\dfrac{2n+1}{2n}$ or $p=2n+1.$
 \end{itemize}

\vskip.5cm
$\bullet$ Let $G=\mathrm{F}_{4(-20)}.$
Here, $K\cong  \mathrm{Spin}(9)$ and $M\cong \mathrm{Spin}(7).$
Moreover, we have $\dim\mathfrak{g}_\alpha= 8$ and  $\dim\mathfrak{g}_{2\alpha}=7$. So, 
$$t_\alpha= 5, \qquad n_\alpha= 2, \qquad\text{and} \qquad \rho_2= 11;$$
hence, for  $\lambda\in\mathbf{R},$ we have
$$\nu_{\lambda, p}= i\lambda+ 11\left(\dfrac{2}{p}-1\right).$$
The abelianization of $M$ is   trivial and so $\widehat{M^{\rm Ab}}=\{1_M\}.$
Moreover, we have
$$11\left|\dfrac{2}{p}-1\right|\geq 5 \Longleftrightarrow p\in (1, \frac{11}{8}] \cup [ \frac{11}{3}, +\infty).$$
and
$$11\left(\dfrac{2}{p}-1\right)+ 5  \in 2 \mathbf{Z}
\Longleftrightarrow p\in \left \{\frac{11}{k}: k\in \mathbf{\NN^*}\right\}.
$$
 It follows from  (\ref{Rep10}) that 
\begin{itemize}
\item  $I_{P,  1_M,  \nu_{\lambda, p}}$  is irreducible 
 for  every $\lambda\in \mathfrak{a}^*\setminus \{0\}$ 
 \item $I_{P,  1_M,  \nu_{0, p}}$  is not irreducible if 
 and only if  $p\in \left\{\dfrac{11}{10 }, \dfrac{11}{9 },  \dfrac{11}{8}, \dfrac{11}{3 }, \dfrac{11}{2 }, 11\right\}.$
 \end{itemize}

In all the cases above, if two representations  $I_{P,  \chi,  \nu}$
and   $I_{P, \chi',  \nu'}$ are irreducible, then they are 
equivalent   if and only if 
 $$  (\chi', \nu)= (\overline{\chi}, \nu)
  \quad  \text{or} \quad  (\chi', \nu')= (\overline{\chi}, -\nu).$$
 Now,
 $$- \nu_{\lambda, p}= -(i\lambda+ \left(\dfrac{2}{p}-1\right)\rho_2= (-i\lambda+ \left(\dfrac{2}{q}-1\right)\rho_2 =\nu_{\lambda, q}.
 $$
 Therefore,  $I_{P,  \chi',  \nu_{\lambda',p'}}$ and   $I_{P,  \chi,  \nu_{\lambda,p}}$
are equivalent if and only if 
 $$  (\chi', \lambda', p')= (\overline{\chi}, \lambda, p)
  \quad  \text{or} \quad  (\chi', \lambda', p')= (\overline{\chi}, -\lambda, q).$$
  
 \section{An example: $SL_2(\mathbf{R})$}
\label{S:Exa1}
Let $G=  \mathrm{SL}_2(\mathbf{R})$, with maximal compact subgroup  $K=SO(2)$. 
The standard minimal parabolic subgroup is
$$
P\,   = \left\{ \begin{pmatrix} a& b 
\\ 0&a^{-1}
\end{pmatrix} : a\in \mathbf{R}, a\neq 0, b\in \mathbf{R}\right\}.
$$
We have $P=MAN$ for 
$$
A \, = \left\{ \begin{pmatrix} a& 0 
\\ 0&a^{-1}
\end{pmatrix} : a\in \mathbf{R}, a>0\right\},
$$
$$
M\,=
 \left\{ \pm\begin{pmatrix} 1& 0\\
 0&1
\end{pmatrix} \right\},
$$
and 
 $$
N \, = \left\{ \begin{pmatrix} 1& b
\\ 
0 & 1 \end{pmatrix} : b\in \mathbf{R}\right\}.
$$
 Identifying  $\mathfrak{a}^*_{\mathbf{C}}$ with $\mathbf{C}$
 as above, we have
 $$\rho_{P, p}= \dfrac{1}{p}$$ for every $p\in (1+\infty).$
We  identify $G/P$ as $G$-space with the real projective line 
$\mathbf{P}(\mathbf{R}),$ with $G$ acting by M\"obius transformations
on this latter space.
The Lebesgue measure $\mu$ is the unique $K$-invariant probability Borel measure
on $\mathbf{P}(\mathbf{R}).$ The measurable space $G/P$ can further be identified
with $\mathbf{R} =\mathbf{P}(\mathbf{R})\setminus \{\infty\}$.

The Radon-Nikodym derivative for $g=\begin{pmatrix} a& b\\ 
c & d \end{pmatrix}$ acting on $\mathbf{R}$ is given by
$$
\frac{dg_*(\mu)}{d\mu}(x)=\dfrac{1}{(cx+d)^2}.
$$
Let $\varepsilon$ denote the non trivial  character of $M$ and 
let $\lambda\in \mathbf{R}.$  The representation $\pi(P,\chi, \lambda, p)$
of $G$ by isometries on $L_p(\mathbf{R}, \mu)$ is defined by 
  $$
  \pi(P, 1_M, \lambda, p)(g) f(x)= |cx+d|^{-i\lambda -\frac{2}{p}} f\left(\frac{ax+b}{cx+d}\right)\quad \text{if} \quad g^{-1}=\begin{pmatrix} a& b
\\ 
c & d \end{pmatrix}
  $$
  and
  $$
  \pi(P, \varepsilon, \lambda, p)(g) f(x)= {\sgn}(cx+d)|cx+d|^{-i\lambda -\frac{2}{p}} f\left(\frac{ax+b}{cx+d}\right)\quad \text{if} \quad g^{-1}=\begin{pmatrix} a& b
\\ 
c & d \end{pmatrix}
  $$
 The Harish-Chandra $(\mathfrak{g}, K)$-module  underlying 
$\pi(P, \varepsilon, \lambda, p)$ is  $I_{P,  \chi, \nu}$ for   
  $$\nu= i\lambda+ \delta_p\rho_{P,2} = i\lambda + \dfrac{2}{p}- 1$$
 Assume that $p\neq 2.$  
 It is well-known    that  $I_{P,  \chi, \nu}$ is irreducible (see e.g. Proposition 1.3.3 in \cite{Vogan};
 observe that the irreducibility of  $I_{P, \varepsilon , \nu}$ depends on the fact that
$\dfrac{2}{p}- 1$ is not an integer). Moreover,  the modules $I_{P, \chi , \nu}$
are pairwise non equivalent for fixed $p.$

In summary, the $L_p$-dual space of $G=\mathrm{SL}_2(\mathbf{R})$ for $p\in(1,+\infty)$ with $p\neq 2$ is
$$
\widehat{G}_{L_p}= \left\{\pi(P, 1_M, \lambda, p) : \lambda\in \mathbf{R}\right\}
\cup \left\{\pi(P, \varepsilon, \lambda, p) : \lambda\in \mathbf{R}\right\} \cup \{1_G\}.
$$

\end{document}